\documentclass[12pt]{amsart}%

\usepackage{amssymb}
\usepackage{amsmath}
\usepackage{amsthm}
\usepackage{amscd,enumitem,graphicx,tikz, pst-node, pst-plot, pstricks}
\usepackage[margin=1.25in]{geometry}%
\usepackage{hyperref}
\usepackage{color}
\usetikzlibrary{patterns,shapes,decorations.pathreplacing}

\theoremstyle{plain}
\newtheorem{theorem}{Theorem}[section]

\newtheorem{corollary}{Corollary}[section]
\newtheorem{lemma}{Lemma}[section]

\theoremstyle{remark}

\newtheorem{remark}{Remark}[section]
\newtheorem{examples}{Examples}[section]
\newtheorem{assumption}{Assumption}[section]

\DeclareMathOperator*{\esssup}{ess\,sup}

\DeclareMathOperator{\diam}{diam}

\DeclareMathOperator{\cu}{curl}
\DeclareMathOperator{\cuc}{curl^*}
\DeclareMathOperator{\dom}{Dom}

\def\Ex{\mathcal E_S}
\def\dd{\mathcal D_S}
\def\Dx{\Delta_S}

\def\domc{\dom(\cu)}
\def\domcc{\dom(\cuc)}
\def\R#1{\ensuremath{\mathbb R^{#1}}}

\newcommand{\xnk}{\ensuremath{S_{n,k}}}
\newcommand{\unk}{\ensuremath{U_{n,k}}}
\newcommand{\fnk}{\ensuremath{f_{n,k}}}
\newcommand{\xxnk}{\ensuremath{x_{n,k}}}
\newcommand{\hn}{\ensuremath{h_{n}}}

\psset{nodesep=5pt,arm=7pt,linearc=5pt}

\begin{document}

\title{Densely defined non-closable curl
 on carpet-like metric measure spaces}
\author{Michael Hinz$^1$}
\address{$^1$ Fakult\"at f\"ur Mathematik, Universit\"at Bielefeld, Postfach 100131, 33501 Bielefeld, Germany}
\email{mhinz@math.uni-bielefeld.de}

\author{Alexander Teplyaev$^2$}
\address{$^2$ Department of Mathematics, University of Connecticut, Storrs, CT 06269-3009 USA}
\email{teplyaev@member.ams.org}
\thanks{$^2$Research supported in part by NSF grant DMS-1613025}

\date{\today}

\begin{abstract}
The paper deals with the possibly degenerate behaviour of the exterior derivative operator defined on $1$-forms on metric measure spaces. The main examples we consider are the non self-similar Sierpinski carpets recently introduced by Mackay, Tyson and Wildrick. Although topologically one-dimensional, they may have positive two-dimensional Lebesgue measure and carry nontrivial $2$-forms. We prove that in this case the curl operator (and therefore also the exterior derivative on $1$-forms) is not closable, and that its adjoint operator has a trivial domain. We also formulate a similar more abstract result. It states that for spaces that are, in a certain way, structurally similar to Sierpinski carpets, the exterior derivative operator taking $1$-forms into $2$-forms cannot be closable if the martingale dimension is larger than one. 

\tableofcontents

\end{abstract}
\maketitle

\section{Introduction}\enlargethispage{.2in}

Our paper is a part of a broader program that aims to connect 
research on derivatives on fractals  (\cite{BBST,BK16,
CGIS12,CS03,CS09,
Hino05,Hino08,Hino10,Hino13,HR14,HKT,HTa,HTfgs5,IRT,
Kajino2012,Kajino13,Ki08,Ku89,St00t,T00} and references therein) 
and on more general regular Dirichlet   spaces \cite{H14a,H14b,HRT}
with  classical and geometric 
 analysis
on metric measure spaces (\cite{Am,BBS,Cheeger,Hei01,Hei07,HeiKoShT,book,KSZ1,KSZ2,KST,KZ,Sh00} and references therein). 
In  our previous article \cite{HT} we showed that 
on certain topologically one-dimensional spaces with 
a strongly
local regular Dirichlet form 
one can prove a natural version of the Hodge theorem for $1$-forms defined in $L^2$-sense: 
the set of harmonic $1$-forms is dense in the orthogonal complement of the exact $1$-forms. 
In this context harmonic $1$-forms appear as limits of $1$-forms that are locally harmonic, i.e.
locally representable as differentials of harmonic functions, \cite[Theorem 4.2]{HT}. This result is complicated 
because in many interesting fractal examples the space of harmonic forms is 
infinite dimensional in a very strong sense, more precisely, its restriction to any non-empty open subset is infinite dimensional.  
In this paper we discuss a question that deals with the exterior derivative operator defined on $1$-forms. It may happen that although the space is topologically one-dimensional, there are nontrivial $2$-forms in the $L^2$-sense. Examples can be found amongst the non-self-similar Sierpinski carpets of positive two-dimensional Lebesgue measure introduced in \cite{MTW} (see Figure~\ref{fig:sc}). 

\begin{figure}[htbp]
	\includegraphics[height=.33\textwidth]{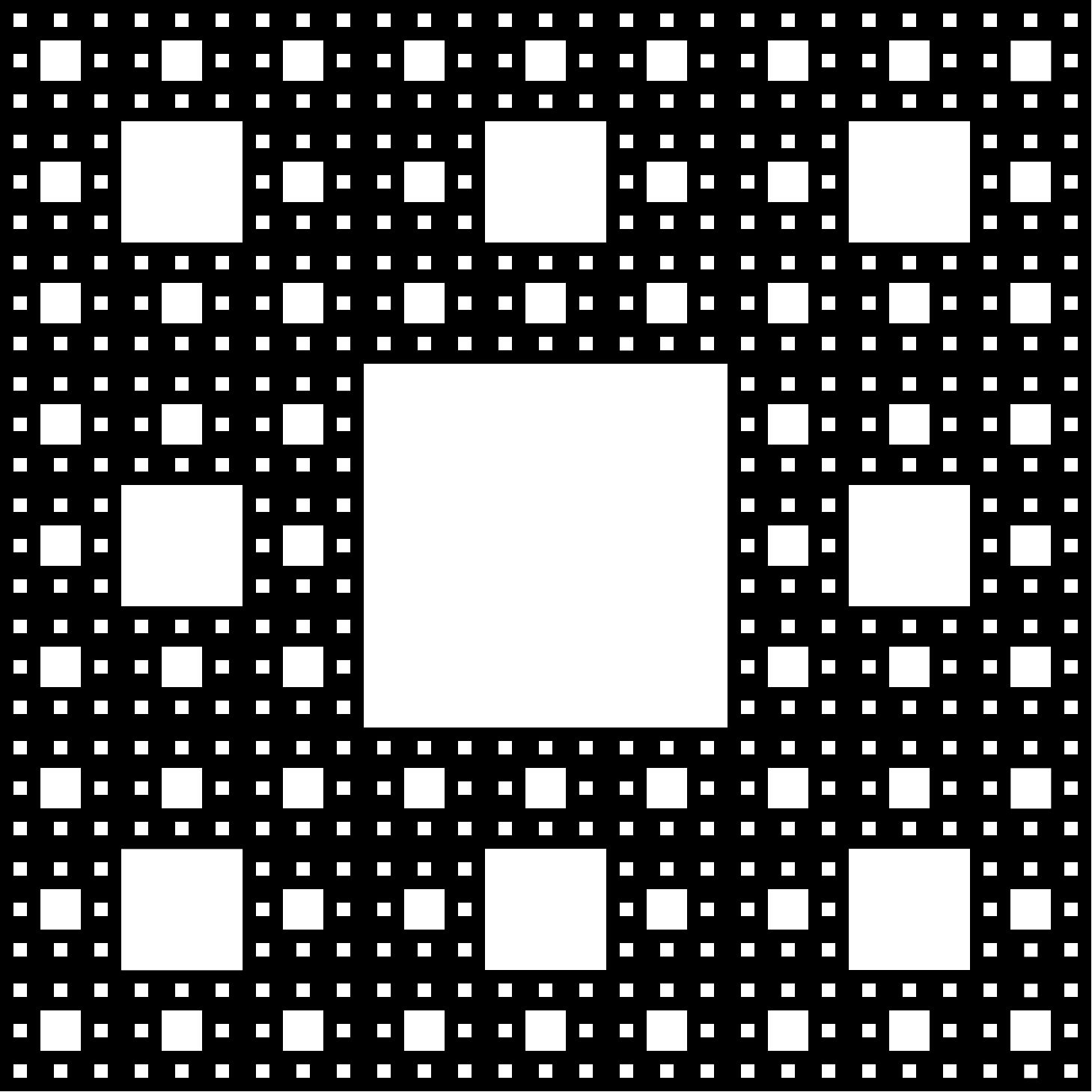}
	\quad\hfil\quad
	\includegraphics[height=.33\textwidth]{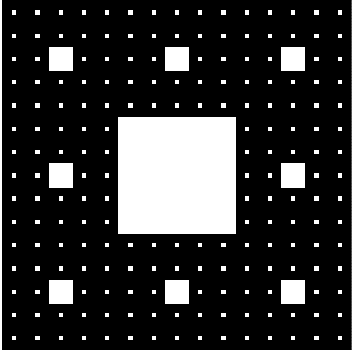}
	\caption{On the left, the standard self-similar 
	Sierpinski carpet. 
	On the right, a non-self-similar Sierpinski carpet
	$S_{(1/3,1/5,1/7,...)}$
	 from the paper of Mackay, Tyson and Wildrick, \cite{MTW} (see page 3 of  \href{http://arxiv.org/pdf/1201.3548v2}{arXiv:1201.3548} for the precise definition)}	
	 \label{fig:sc}
\end{figure}

On such a carpet $S$ we can consider the classical notions of energy and exterior derivation by restricting smooth functions and forms from $\mathbb{R}^2$ to $S$. In particular, applying the (classical) exterior derivation to smooth $1$-forms on $\mathbb{R}^2$ and restricting the resulting $2$-forms to $S$, we can observe the existence of nonzero square integrable $2$-forms on $S$. On the other hand, if the holes in the carpet are not too small, then \cite[Theorem 4.2]{HT} applies (see Remark \ref{R:connect} below) and tells that within the $L^2$-space of $1$-forms the locally harmonic $1$-forms are dense in the orthogonal complement of the exact $1$-forms. Here we call a square integrable $1$-form $\omega$ locally harmonic if we can find a finite open cover $\left\lbrace U_\alpha\right\rbrace_{\alpha\in J}$ of $S$ and functions $h_\alpha$, $\alpha\in J$, harmonic in the Dirichlet form sense, such that for any $\alpha\in J$ we have $\omega \mathbf{1}_{U_\alpha}=dh_\alpha \mathbf{1}_{U_\alpha}$, where $d$ denotes the exterior derivation. Now assume there is a closed extension of the classical exterior derivation $d$ on $1$-forms to a closed and densely defined unbounded linear operator on the $L^2$-space of $1$-forms. Since $d\circ d=0$, its application to a locally harmonic $1$-form $\omega$ would give $d\omega=0$. Therefore $d$, seen in the $L^2$-sense, would give zero on its entire domain, seemingly contradicting the existence of nonzero square integrable $2$-forms. However, what we observe in this paper is that, roughly speaking, if the holes in a generalized carpet are not too small, then the exterior derivation on $1$-forms is not closable in the $L^2$-sense. So there is no contradiction, but an interesting difference between the classical and the $L^2$-formulation.

On carpets of positive two-dimensional Lebesgue the exterior derivative on $1$-forms can be expressed using the 
curl operator.
Since our main motivation comes from quantum physics, especially \cite[pages 850-862]{AkkermansMallick},
the closability of $\cu$ and the domain of definition of the adjoint operator 
$\cu^*$, seen as operators between Hilbert spaces, are of special interest. If the diameter of the holes in the carpet does not decrease too rapidly and the domain of $\cu$ contains all smooth vector fields, we can observe the non-closability of $\cu$, see Corollary \ref{C:notclosable}. However, in this situation we can even state a stronger result, namely that the domain of its adjoint $\cu^\ast$ must be trivial, Theorem \ref{thm01}. The proof of this theorem is elementary and graphical, based on the specific structure of generalized non-self-similar Sierpinski carpets. 

We also prove a similar result in a more abstract setup. Suppose that $X$ is a compact metric space, $\mu$ a finite Radon measure on $X$ with full support and $(\mathcal{E},\mathcal{F})$ a strongly local regular Dirichlet form  on $L^2(X,\mu)$. We additionally assume that we are given a finite energy dominant measure $m$ for $(\mathcal{E},\mathcal{F})$ and an algebra $\mathcal{A}$ of functions with $m$-essentially bounded energy densities that provides a special standard core for $(\mathcal{E},\mathcal{F})$. It is always possible to find such measures and algebras. The $L^2$-space of $1$-forms can be defined without ambiguity. Using the algebra $\mathcal{A}$ we can introduce an $L^2$-space of $2$-forms with respect to the measure $m$. The $m$-essential supremum of the dimensions of the corresponding abstract cotangent spaces, clearly integer valued, is referred to as the \emph{martingale dimension}, \cite{Hino10}. The second version of our result is stated for spaces that, in a certain way, have a similar structure as Sierpinski carpets, see Assumption \ref{A:carpetlike}. Theorem \ref{T:general} states that under this assumption either the martingale dimension equals one or the exterior derivation on $1$-forms, considered on a certain dense initial domain, is not closable.

The existence of a non-closable curl or, respectively, exterior derivation, may be seen as part of a discussion about the role of dimensions of metric measure spaces carrying a diffusion. The spectral dimension is most significant for elliptic and parabolic equations for scalar functions. For vector equations we already observed in \cite{HT} that other structural properties may be relevant. There we promoted a version of a Navier-Stokes system on topologically one-dimensional spaces and observed its simplification to an Euler type equation 
that has infinitely many nontrivial steady state unique weak solutions. We formulated the system in a way that assumes that there are no nontrivial $2$-forms. (Note, however, that all other results in \cite{HT} are absolutely independent of this question.)  

There exists an extensive literature 
that establishes the relation between 
function spaces 
 on metric measure spaces and the theory of Dirichlet forms. 
In particular the reader can consult the papers 
\cite{BBKT,BBS,KSZ1,KSZ2,KST,KZ} 
and references therein. Our current paper is a step in a long-term program 
(see 
\cite{H14a,H14b,HR14,HKTg,HKT,HRT,HTa,HT,HTfgs5,IRT}) 
to develop parts of differential  
geometry and their applications to mathematical physics 
(see \cite{ADT,ADT2010PRL,ADTV})  for spaces 
that carry diffusion processes but no other smooth structure. 
Our approach is somewhat complementary to the celebrated works 
\cite
{Cheeger,Hei01,Hei07} 
because, although our spaces are metrizable, 
we do not use any particular metric in an 
essential way, and we do not use functional inequalities. A  different approach to the differential forms and the 
Hodge - de Rham theory on fractal graphs and fractals  is introduced in 
 \cite{Str-hodge}. It will be a subject of future studies to establish the  
 connection between  
 \cite{Str-hodge} and our work. 
 
For symmetric bilinear expressions $\mathcal{B}(f,g)$ of two arguments $f$ and $g$ we use the notation $\mathcal{B}(f):=\mathcal{B}(f,f)$ for the associated quadratic expression $\mathcal{B}$. Recall also that given a quadratic expression $\mathcal{Q}(f)$ of one argument $f$, polarization yields a symmetric bilinear expression $\mathcal{Q}(f,g)$ of two arguments.

\subsection{Acknowledgment} 
The authors are very grateful to Eric Akkermans and Gerald Dunne for interesting and helpful discussions and to Naotaka Kajino and Jun Kigami for pointing out some necessary corrections.


\section{Sierpinski carpets of positive two-dimensional Lebesgue measure}\label{S:sc}

In this section we discuss prototype examples given by generalized Sierpinski carpets and corresponding restrictions of the classical energy form and the classical curl on $\mathbb{R}^2$.

 We recall a construction studied in \cite[Section 1]{MTW}. Let $\mathbf{a}=(a_1,a_2,\dots)$  be a sequence of positive reals such that for any $i$ the number $\frac{1}{a_i}$ is an odd integer strictly greater than one. Let $S_{\mathbf{a},0}:=[0,1]^2$ be the unit square, we can rewrite it as the union of $(\frac{1}{a_1})^2$ congruent closed subsquares of side length $a_1$ that touch only at their boundaries. Let $\mathcal{T}_1$ be the family of all such subsquares except the central one and put $S_{\mathbf{a},1}:=\bigcup_{T\in\mathcal{T}_1} T$. Next, let $\mathcal{T}_2$ be the family of all congruent closed subsquares of side length $a_1a_2$, touching only at the boundaries, obtained by subdividing each element of $\mathcal{T}_1$ in a similar way as $[0,1]^2$ and discarding the central squares, respectively. Set $S_{\mathbf{a},2}:=\bigcup_{T\in \mathcal{T}_2} T$. Further iteration of this construction process yields a decreasing sequence $\left\lbrace S_{\mathbf{a},m}\right\rbrace_{m\geq 0}$ of nonempty compact sets $S_{\mathbf{a},m}$. To 
\[S_{\mathbf{a}}:=\bigcap_{m\geq 0} S_{\mathbf{a},m}\] 
we refer as the \emph{generalized Sierpinski carpet associated with the sequence $\mathbf{a}$}.
For the constant sequence $\mathbf{a}=(\frac13, \frac13, \dots)$ one obtains the classical Sierpinski carpet. Any $S_{\mathbf{a}}$ is a compact subset of $\mathbb{R}^2$, and we consider $S_{\mathbf{a}}$ with the relative topology (induced topology). Then any nonempty open subset of $S_{\mathbf{a}}$ is topologically one-dimensional. If $\mathbf{a}\in l_2$, then any nonempty open subset of $S_{\mathbf{a}}$ has positive two-dimensional Lebesgue measure $\lambda^2$, see \cite[Proposition 3.1 (iv)]{MTW}, and the restriction of $\lambda^2$ to $S_{\mathbf{a}}$ is Ahlfors $2$-regular.

\begin{examples}\label{Ex:simpleex}
In Figure \ref{fig:sc} on the right hand side we have $a_n=\frac{1}{2n+1}\in l_2$, and so any nonempty open subset of the carpet $S_{\mathbf{a}}$ associated with this sequence $\mathbf{a}=(a_n)_{n\geq 1}$ has positive $\lambda^2$ measure. On the left hand side of this figure 
we have the standard self-similar 
	Sierpinski carpet, which has zero Lebesgue measure and  
	$a_n\equiv\frac{1}{3}\notin l_2$.
\end{examples}

Let $\mathbf{a}\in l_2$ be fixed and consider $S:=S_{\mathbf{a}}$. We write $L^2(S)$ for the space of (classes of) functions on $S$ that are square integrable with respect to $\lambda^2$. The restriction to $S$ of the usual Dirichlet integral, defined for any $f\in C^1(\mathbb R^2)$ by 
\[\Ex(f) =\int_{S} |\nabla f(x,y)|^2d(x,y),\]
extends to a local regular Dirichlet form $(\Ex, \dd)$ on $L^2(S)$. 

\begin{remark} The last statement can be seen following arguments similar to the ones given in \cite[p. 246-247]{KST}: For any rectifiable curve $\gamma:[a,b]\to S$ and any $f\in C^1(\mathbb{R}^2)$ the line integral $\int_\gamma (\nabla f)ds$ of its gradient $\nabla f$ along $\gamma$ equals $f(\gamma(b))-f(\gamma(b))$. This implies that for $f\in C^1(\mathbb{R}^2)$ the function
$|\nabla f|$ is the minimal upper gradient of $f:S\to \mathbb{R}$, and since $|\nabla f|\in L^2(S)$, any function $f\in C^1(\mathbb{R}^2)$ is a member of the (Newtonian) Sobolev space $N^{1,2}(S)$ in the sense of \cite[Definition 2.5]{Sh00}, see also \cite[Definition 3.1]{HeiKoShT}. Moreover, there is some $c>0$ such that 
\[c\left\|f\right\|_{N^{1,2}(S)}^2\leq \Ex(f)+\left\|f\right\|_{L^2(S)}^2\leq c^{-1}\left\|f\right\|_{N^{1,2}(S)}^2,\ \ f\in C^1(\mathbb{R}^2).\]
This implies the closability of $(\Ex, C^1(\mathbb{R}^2))$ in $L^2(S)$. Its closure $(\Ex, \dd)$ is regular, note that by Tietze's extension lemma and local polynomial approximation the $C^1(\mathbb{R}^2)$-functions are dense in $C(S)$.
\end{remark}

\begin{remark} Seen as a subset of $\mathbb{R}^2$, the set $S$ is fairly complicated. In particular, the complement $S^c=\R2\backslash S$ of $S$ is an open set that is everywhere dense, but does not have full Lebesgue measure $\lambda^2$ in any neighborhood of any point of $S$. The topological boundary  of $S$ in the topology of $\mathbb R^2$ coincides with $S$, that is, every point of $S$ is a boundary point if we consider the usual $\R2$ topology.

Endowed with the induced topology, $S$ has no intrinsic boundary. Note that, because of the remark in the next paragraph, there are no boundary terms in the  Green's formula \eqref{eGreen} on $S$. 
\end{remark}

\begin{remark}
The existence of the closed energy form $(\Ex, \dd)$ has many remarkable consequences, including the existence of the unique non-negative self-adjoint Laplacian $\Dx$ on $S$, corresponding to $\Ex$ and $\lambda^2$ restricted to  $S$. 
This Laplacian $\Dx$ can be defined weakly via the usual relation 
\begin{equation}
\Ex(f,g) =\int_{S}  f(x,y)\Dx \:g(x,y)d(x,y)	 
\label{eGreen}
\end{equation}
for $g$ in the domain of $\Dx$ and $f\in \dd$. Although this is formula can be considered a version of the classical Green's formula, the Laplacian $\Dx$ is not given as the sum of the second derivatives of twice continuously differentiable functions.

Note also that although the domain $\dd$ of the Dirichlet form 
$\Ex$ can be identified with a version of a Sobolev space, 
the domain of the Laplacian  $\Dx$ maybe more complicated and is only weakly defined. 
In particular, the are no reasons to think that the domain of  $\Dx$ contains any 
non-constant $C^2(\mathbb R^2)$ functions (see \cite{BBST} for some related questions). 
This is because in an open region with smooth boundary the domain of the Laplacian consists of functions with vanishing normal derivative at the boundary, but in the case of $S$ with a dense complement $S^c=\R2\backslash S$, this would mean a really singular behavior of the derivative. 
In fact it is natural to conjecture that the 
  domain of  $\Dx$ contains no 
non-constant $C^1(\mathbb R^2)$ functions restricted to $S$. 
\end{remark}

In \R2\  the curl is 
defined for a smooth enough vector field $u=(u_1,u_2):\mathbb R^2\to\mathbb R^2$ 
by
$$\cu u (x,y)= \frac{\partial u_2(x,y)}{\partial x}-\frac{\partial u_1(x,y)}{\partial y}.$$ 
Of course we have in mind the usual relation $\cu u=\nabla\times u$ 
which can be justified in the sense that 
one can consider a three-dimensional vector field 
$(u_1,u_2,0)$ and compute, in three dimensions, that 
$\nabla\times(u_1,u_2,0)=(0,0,\cu u)$. Thus, $\cu u$ is the third component of 
the three dimensional vector field $\nabla\times(u_1,u_2,0)$. 
At the same time the $\cu$ operator can be described using the 
notion of exterior derivative $\text d$ of 
differential 1-forms in the sense that we have 
$$\text d(u_1(x,y)\text dx+u_2(x,y)\text dy)=\cu u (x,y)\text dx\wedge\text  dy$$ 
When discussing the situation on $\R2$, we can identify vector fields $u=(u_1, u_2)$ and $1$-forms $u_1dx+u_2dy$. 

For a vector field $u=(u_1,u_2)$ with $u_1,u_2\in C^1(\mathbb{R}^2)$ the function $\cu u$ is continuous, its pointwise restriction to $S$ makes sense and is a member of $L^2(S)$. Therefore $(\cu, C^1(\mathbb{R}^2))$ may be seen as a densely defined unbounded operator from $L_2(S,\R{2})$ into $L^2(S)=L^2(S,\R{})$. We slightly reformulate this situation by endowing $\cu$ with a dense domain $\domc$ and asking what happens to the domain $\domcc$ of its adjoint $\cu^\ast$. This of course also decides whether $(\cu, \domc)$ is closable or not.
The following theorem is a version of our main result in the \R2\ case.

\begin{theorem}\label{thm01}
Let $\mathbf{a}\in l_2$ be a sequence such that 
\begin{equation}\label{E:condan}
\lim_n\frac{a_1\cdots a_{n-1}}{a_n}=0
\end{equation}
and consider $S=S_{\mathbf{a}}$. If $\domc \subset L_2(S,\R{2})$ 
contains all smooth 
vector fields, then $\domcc\subset L^2(S)$ is trivial, $\domcc=\{0\} $. 
\end{theorem}

The condition that $\mathbf{a}\in l_2$ forces the size of the holes to decrease sufficiently fast, condition (\ref{E:condan}) on the other hand requires them not to decrease extremely fast.

\begin{examples}
If $a_n=\frac{1}{2n+1}$ for all $n$ as in Examples \ref{Ex:simpleex}, then $\mathbf{a}\in l_2$ and (\ref{E:condan}) is satisfied.
\end{examples}
\begin{remark}
We believe that the condition \ref{Ex:simpleex} can be relaxed substantially, in particular using the methods of \cite{Str-graph}, but we present here only a simplified construction 
because it is sufficient for our purpose. 
\end{remark}

The next Corollary is immediate from \cite[Theorem VIII.1]{RS}.

\begin{corollary}\label{C:notclosable} 
Under conditions of Theorem~\ref{thm01} $\cu$ is not closable. 
\end{corollary}


\section{Proof of Theorem \ref{thm01}}

We give a proof by contradiction. 
Suppose $u\in\domcc\subset L^2(S)$, $u\neq0$ and $\cuc u=w\in L^2(S,\R{2})$. 
Then there is a smooth function $f$ approximating $u$, and in particular 
$\left\langle u,f\right\rangle_{L_2(S)}>0$. 
Our main claim is that we can 
  construct a sequence of smooth vector fields $v_n$ 
that satisfies the following two conditions:
\begin{enumerate}[label=(\alph*)]  
	\item\label{c01a} $\lim_{n\to\infty}\cu v_n=f$ in $L^2(S)$; 
	\item\label{c01b} $\lim_{n\to\infty}v_n=0$ in $L_2(S,\mathbb{R}^2)$. 
\end{enumerate}
If these two items are satisfied, we observe the contradiction
\begin{equation*}
	0=\lim_{n\to\infty}\left\langle w,v_n\right\rangle_{L_2(S,\mathbb{R}^2)}=\lim_{n\to\infty}\left\langle \cuc u,v_n\right\rangle_{L_2(S,\mathbb{R}^2)}=\lim_{n\to\infty}\left\langle u,\cu v_n\right\rangle_{L_2(S)}=\left\langle u,f\right\rangle_{L_2(S)}>0,
\end{equation*}
and this completes the proof. 

The main technical ingredient is to demonstrate that there exists a sequence of smooth vector fields $v_n$ satisfying \ref{c01a} and \ref{c01b} above. 
To construct $v_n$, we first cover $S$ by the compact subsets $\xnk$ (in relative topology)
obtained by taking parallels to the coordinate axes through the midpoints of all holes of size 
\begin{equation}\label{E:deltan}
\delta_n:=a_1\cdots a_n
\end{equation}
that are created at stage $n$ of the construction of $S_{\mathbf{a}}$, see Figure \ref{F:Snk}. They intersect each other over Cantor sets
, and the diameter of each \xnk\ is bounded by $\sqrt{2}\delta_{n-1}$. 

\begin{figure}
   \centering 
\begin{tikzpicture}
\draw[black, very thick] (0,0)rectangle (7.5,7.5);
\draw[black, very thick] (3,3)rectangle (4.5,4.5);
\draw[black, very thick] (3.6,0.6) rectangle (3.9,0.9);
\draw[black, very thick] (3.6,2.1) rectangle (3.9,2.4);
\draw[black, very thick] (0.6,0.6) rectangle (0.9,0.9);
\draw[black, very thick] (0.6,2.1) rectangle (0.9,2.4);
\draw[black, very thick] (0.6,3.6) rectangle (0.9,3.9);
\draw[black, very thick] (2.1,0.6) rectangle (2.4,0.9);
\draw[black, very thick] (2.1,2.1) rectangle (2.4,2.4);
\draw[black, very thick] (2.1,3.6) rectangle (2.4,3.9);
\draw[green, dashed] (3.75,0)--(3.75,7.5);
\draw[green, dashed] (0, 3.75)--(7.5,3.75);
\draw[red, dashed] (0, 0.75)--(7.5,0.75);
\draw[red, dashed] (0, 2.25)--(7.5,2.25);
\draw[red, dashed] (0.75, 0)--(0.75,7.5);
\draw[red, dashed] (2.25, 0)--(2.25,7.5);
\node at (1.5,1.5) {\xnk};
\end{tikzpicture}
\caption{The sets $S_{n,k}$.}\label{F:Snk}
\end{figure}
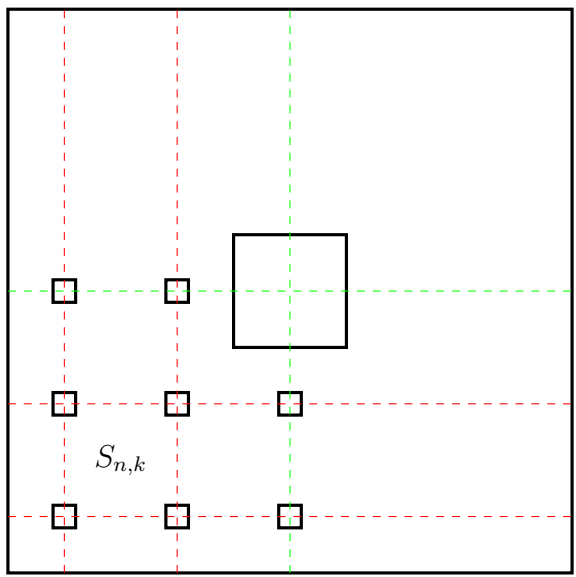
By choosing sufficiently small neighborhoods $\unk$ of the boundaries of 
the sets $\xnk$ we can construct a sequence of energy finite functions $g_n$ such that:
\begin{enumerate}[label=(\roman*)]
\item\label{c01c}
$\nabla g_n$ is arbitrarily close to the vector field $(0,1)$ in 
$L^2(S,\R2)$;
\item\label{c01d} each $g_n$ is locally constant on each \unk. 
\end{enumerate}

For fixed $n$ consider the horizontal Cantor sets that arise as parts of the boundaries of the sets $\xnk$ and are aligned parallel to the $x$-axis. Consider the union $F_n$ of their vertical parallel sets with distance $\delta_n/2$, that is, the union of the horizontal strips of height $\delta_n$ and with the corresponding Cantor set in the middle. Let now $\varphi_n=\varphi_n(x,y)$ be a continuous function on $S$ that is constant in $x$ on $S$, constant in $y$ on $F_n$ and that on each connected component of $S\setminus F_n$ differs from $g(x,y):=y$ by an additive constant, see Figure \ref{F:phin}. Clearly each $\varphi_n$ is the restriction to $S$ of a Lipschitz function and therefore has finite energy. The functions $\varphi_n$ approximate $g$ in energy,
\[\lim_n \Ex(g-\varphi_n)=\lim_n\int_{F_n}(\nabla(g-\varphi_n))^2\:d\lambda^2=\lim_n\lambda^2(F_n)\leq\lim_n\frac{\delta_n}{a_1\cdots a_{n-1}}=\lim_n a_n=0.\]

\begin{figure}
   \centering 
\begin{tikzpicture}
\fill[fill=yellow] (3,1.2)rectangle(4.2,1.8);
\fill[fill=yellow] (3,4.2)rectangle(4.2,4.8);
\fill[fill=yellow] (4.8,1.2)rectangle(6,1.8);
\fill[fill=yellow] (4.8,4.2)rectangle(6,4.8);
\draw[blue] (3,1.2)rectangle(4.2,1.8);
\draw[blue] (3,4.2)rectangle(4.2,4.8);
\draw[blue] (4.8,1.2)rectangle(6,1.8);
\draw[blue] (4.8,4.2)rectangle(6,4.8);
\draw[black, very thick] (3,3)rectangle (6, 6);
\draw[black, very thick] (3,0)rectangle (6,3); 
\draw[black, very thick] (4.2,1.2) rectangle (4.8,1.8);
\draw[black, very thick] (4.2,4.2) rectangle (4.8,4.8);
\draw[red, dashed] (3,1.5)--(4.2,1.5);
\draw[red, dashed] (3,4.5)--(4.2,4.5);
\draw[red, dashed] (4.8,1.5)--(6,1.5);
\draw[red, dashed] (4.8,4.5)--(6,4.5);
\draw[black, very thick] (12,0)--(12,6);
\draw[blue] (12,0)--(10.8,1.2);
\draw[blue] (10.8,1.2)--(10.8,1.8);
\draw[blue] (10.8,1.8)--(8.4,4.2);
\draw[blue] (8.4,4.2)--(8.4,4.8);
\draw[blue] (8.4,4.8)--(7.2,6);
\end{tikzpicture}
\caption{The functions $\varphi_n$.}\label{F:phin}
\end{figure}
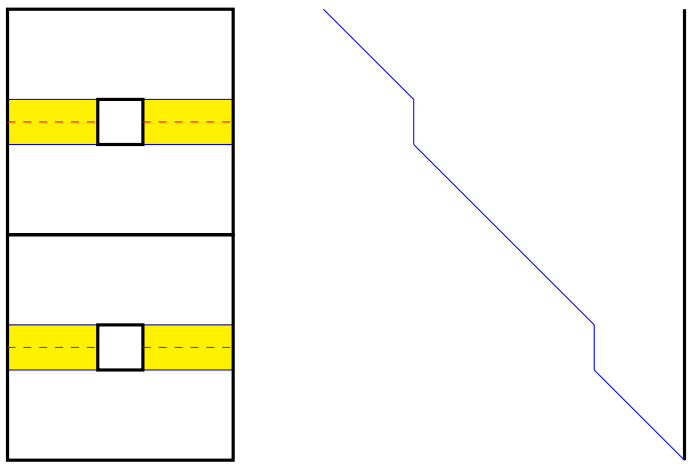
For fixed $n$ consider now the vertical Cantor sets that arise as parts of the boundaries of the sets $\xnk$ and are aligned parallel to the $y$-axis. Connect two vertically adjacent holes of $S$ of side length at least $\delta_n$ by rectangles with horizontal side lengths $\delta_n$ and vertical side lengths $\varepsilon_n$. Further, consider the isosceles trapezoids created by connecting the upper horizontal edges of the lower holes with the centered parts of length $\delta_n/2$ of the lower horizontal edges of the upper holes. Then the vertical Cantor sets are located on the (joint) vertical symmetry axis of these rectangles and trapezoids. Now let $\psi_n$ the function on $[0,1]$ created by putting little tents 
over each rectangle such that $\varphi_n$ is zero on the left, right and lower edge of each rectangle, has value
\begin{equation}\label{E:epsn}
\varepsilon_n:=(1-a_n)(a_1\cdots a_{n-1})
\end{equation}
on the entire upper (short) edge of the trapezoid and is linear in between. For pieces that touch the boundary of $S$ construct $\varphi_n$ as if the interior of $S$ was mirrored to the outside. See Figure \ref{F:psin}. The function $\psi_n$ is nonnegative, Lipschitz and supported in the union of all these small rectangles. The number of such rectangles is bounded by $\frac{2}{a_1\cdots a_{n-1}}$. The restriction to $S$ of these functions has finite energy: Of each tent the energy form $\Ex$ sees only the part over the rectangles between the holes, not the part on the hole. On each trapezoid accommodating one of the Cantor sets between two holes (or a hole and the boundary of $S$) the function $\psi_n$ has slope one in $y$-direction and is constant in $x$-direction, outside the trapezoid it is constant in $y$-direction and has slope $\pm 4\varepsilon_n/\delta_n$ in $x$-direction. Therefore a typical tent contributes the energy 
\[\frac34 \varepsilon_n\delta_n+8\frac{\varepsilon_n^3}{\delta_n},\]
and for situations close to the boundary this value serves as an upper bound. Using (\ref{E:deltan}) and (\ref{E:epsn}) together with the fact that there are less than $\frac{2}{a_1\cdots a_{n-1}}$ tents, we obtain 
\[\Ex(\psi_n)\leq \frac32 (1-a_n) a_1\cdots a_n +16 (1-a_n)^3\:\frac{a_1\cdots a_{n-1}}{a_n},\]
and by (\ref{E:condan}) it follows that $\lim_n\Ex(\psi_n)=0$. The functions $g_n:=\varphi_n-\psi_n$ now satisfy
\[\lim_n\Ex(g-g_n)^{1/2}\leq \lim_n\Ex(g-\varphi_n)^{1/2}+\lim_n\Ex(\psi_n)^{1/2}=0,\]
what is \ref{c01c}. Each $g_n$ is locally constant on the neighborhood $\unk$ of $\xnk$ consisting of two rectangles and two trapezoids (with obvious modifications at the boundary of $S$), see Figure \ref{F:unk}, what shows \ref{c01d}.

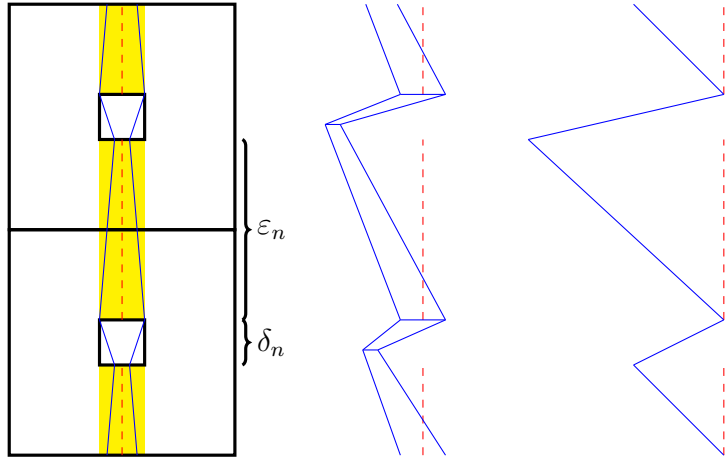
\begin{figure}
   \centering 
\begin{tikzpicture}
\fill[fill=yellow] (4.2,0)rectangle(4.8,1.2);
\fill[fill=yellow] (4.2,1.8)rectangle(4.8,4.2);
\fill[fill=yellow] (4.2,4.8)rectangle(4.8,6);
\draw[black, very thick] (3,3)rectangle (6, 6);
\draw[black, very thick] (3,0)rectangle (6,3); 
\draw[black, very thick] (4.2,1.2) rectangle (4.8,1.8);
\draw[black, very thick] (4.2,4.2) rectangle (4.8,4.8);
\draw[black, very thick, decoration={
        brace,
        mirror,
        raise=0.1cm
    },
    decorate] (6,1.8)--(6,4.2) node[midway,right]{\(\textcolor{black}{\ \varepsilon_n}\)};
\draw[black, very thick, decoration={
        brace,
        mirror,
        raise=0.1cm
    },
    decorate] (6,1.2)--(6,1.8) node[midway,right]{\(\textcolor{black}{\ \delta_n}\)};
\draw[red, dashed] (4.5,0)--(4.5,1.2);
\draw[red, dashed] (4.5,1.8)--(4.5,4.2);
\draw[red, dashed] (4.5,4.8)--(4.5,6);
\draw[blue] (4.8,1.8)--(4.6,4.2);
\draw[blue] (4.2,1.8)--(4.4,4.2);
\draw[blue] (4.6,4.2)--(4.8,4.8);
\draw[blue] (4.4,4.2)--(4.2,4.8);

\draw[blue] (4.8,4.8)--(4.7,6);
\draw[blue] (4.2,4.8)--(4.3,6);
\draw[blue] (4.3,0)--(4.4,1.2);
\draw[blue] (4.7,0)--(4.6,1.2);
\draw[blue] (4.4,1.2)--(4.2,1.8);
\draw[blue] (4.6,1.2)--(4.8,1.8);
\draw[red, dashed] (8.5,0)--(8.5,1.2);
\draw[red, dashed] (8.5,1.8)--(8.5,4.2);
\draw[red, dashed] (8.5,4.8)--(8.5,6);
\draw[blue] (8.2,1.8)--(7.2,4.4);
\draw[blue] (8.8,1.8)--(7.4,4.4);
\draw[blue] (7.4,4.4)--(8.8,4.8);
\draw[blue] (7.2,4.4)--(8.2,4.8);
\draw[blue] (8.2,4.8)--(8.8,4.8);
\draw[blue] (7.2,4.4)--(7.4,4.4);
\draw[blue] (8.2,1.8)--(8.8,1.8);
\draw[blue] (8.8,4.8)--(8.16,6);
\draw[blue] (8.2,4.8)--(7.74,6);
\draw[blue] (8.8,0)--(7.9,1.4);
\draw[blue] (8.2,0)--(7.7,1.4);
\draw[blue] (7.7,1.4)--(8.2,1.8);
\draw[blue] (7.9,1.4)--(8.8,1.8);
\draw[blue] (7.7,1.4)--(7.9,1.4);
\draw[red, dashed] (12.5,0)--(12.5,1.2);
\draw[red, dashed] (12.5,1.8)--(12.5,4.2);
\draw[red, dashed] (12.5,4.8)--(12.5,6);
\draw[blue] (12.5,0)--(11.3,1.2);
\draw[blue] (11.3,1.2)--(12.5,1.8);
\draw[blue] (12.5,1.8)--(9.9,4.2);
\draw[blue] (9.9,4.2)--(12.5,4.8);
\draw[blue] (12.5,4.8)--(11.3,6);
\end{tikzpicture}
\caption{The functions $\psi_n$.}\label{F:psin}
\end{figure}

\begin{figure}
   \centering 
\begin{tikzpicture}
\fill[fill=yellow] (0.9,0.6)rectangle(2.1,0.9);
\fill[fill=yellow] (0.9,2.1)rectangle(2.1,2.4);
\fill[fill=yellow] (0.6,0.9)--(0.675,2.1)--(0.825,2.1)--(0.9,0.9)--cycle;
\fill[fill=yellow] (2.1,0.9)--(2.175,2.1)--(2.325,2.1)--(2.4,0.9)--cycle;
\draw[red,dashed] (0.9,0.75)--(2.1, 0.75);
\draw[red,dashed] (0.75,0.9)--(0.75, 2.1);
\draw[red,dashed] (0.9,2.25)--(2.1, 2.25);
\draw[red,dashed] (2.25,0.9)--(2.25, 2.1);
\draw[black, very thick] (2.1,0.6) rectangle (2.4,0.9);
\draw[black, very thick] (2.1,2.1) rectangle (2.4,2.4);
\draw[black, very thick] (0.6,0.6) rectangle (0.9,0.9);
\draw[black, very thick] (0.6,2.1) rectangle (0.9,2.4);
\node at (1.5,1.5) {\xnk};
\end{tikzpicture}
\caption{The neighborhoods $\unk$ of the boundary of $\xnk$.}\label{F:unk}
\end{figure}
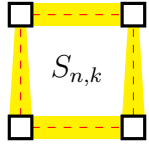

Let \fnk\ be one of the values of the function $f$ on \xnk, 
and let 
\xxnk\ be one of the values of the $x$ coordinate  on \xnk. 
There exists a sequence of smooth functions \hn\ 
such that $$\|\hn\|_{\sup}\leqslant\frac1n\|f\|_{\sup}$$ 
and on each set $\xnk\backslash\unk$ we have 
$$\hn(x,y)=\fnk(x-\xxnk).$$
Then we define 
$$v_n=\hn\nabla g_n.$$
It is evident 
that \ref{c01b} is satisfied 
and that 
$$\cu v_n=\fnk(\nabla g_n)_2+\hn\cu\nabla g_n=\fnk(\nabla g_n)_2$$
where $(\nabla g_n)_2=\frac{g_n(x,y)}{\partial y}$ denotes the second component 
of the vector field $\nabla g_n$. 
Note that by our construction this second component 
$(\nabla g_n)_2$ converges in $L^2(S)$ 
to a function that is identically equal to 1 on $S$, and so 
$\cu v_n$ converges to $f$ in $L^2(S)$. This concludes the proof of Theorem \ref{thm01}.

The key element of our construction is the fact that 
$\nabla g_n$ vanishes on each \unk\ and so we do not have to 
analyze the derivatives of $h_n$ on \unk, although 
one can see that these derivatives can not be small. 


\section{Local Dirichlet forms on carpet-like spaces%
}

The proof of Theorem \ref{thm01} in the preceding section is based on the specific structure of generalized Sierpinski carpets and the coordinate structure of $\mathbb{R}^2$. Under some abstract conditions we can implement its main idea in a more general setup.

Let $(X,\varrho)$ be a compact metric space, $\mu$ a finite Radon measure on $X$  with full support and $(\mathcal{E},\mathcal{F})$ a strongly local regular symmetric Dirichlet form on $L^2(X,\mu)$. We write 
\[\mathcal{E}_1(f,g):=\mathcal{E}(f,g)+\left\langle f,g\right\rangle_{L^2(X,\mu)},\ \ f,g\in\mathcal{F},\] 
for the scalar product in the Hilbert space $\mathcal{F}$. The space $\mathcal{F}\cap C(X)$ is an algebra containing the constant $\mathbf{1}$. The \emph{energy measure} $\Gamma(f,g)$ of $f,g\in \mathcal{F}\cap C(X)$ is defined by
\[2\int_X h d\Gamma(f,g)=\mathcal{E}(fh,g)+\mathcal{E}(gh,f)-\mathcal{E}(fg,h), \ \ h\in\mathcal{F}\cap C(X),\]
and approximation defines $\Gamma(f,g)$ for general $f,g \in\mathcal{F}$. See for instance \cite[Section 3.2]{FOT94}.  Again we write $\mathcal{E}(f)$ for $\mathcal{E}(f,f)$ and $\Gamma(f)$ for $\Gamma(f,f)$. 

\begin{examples}\label{Ex:rmf}
If $M$ is a compact Riemannian manifold, $dvol$ denotes the Riemannian volume and $d$ the exterior derivative, the symmetric bilinear form given by
\[\mathcal{E}(f,g)=\int_M \left\langle d_xf, d_xg\right\rangle_{T_x^\ast M}dvol(x),\ \ f,g\in W^{1,2}(M),\] 
is a strongly local regular Dirichlet form on $L^2(M, dvol)$. The energy measures have densities $x\mapsto \left\langle d_xf,d_xg\right\rangle_{T^\ast_x M}$ with respect to $dvol$.
\end{examples}

Subsequent constructions will be based on Hilbert spaces $\mathcal{H}_x$ that substitute for the cotangent spaces $T_x^\ast M$, they can be constructed using energy densities. If, as in Example \ref{Ex:rmf}, all energy measures are absolutely continuous with respect to $\mu$ one says that $(\mathcal{E},\mathcal{F})$ admits a \emph{carr\'e du champ}, \cite{BH91}, and the corresponding densities can be used. However, for many local Dirichlet forms on fractal spaces the measures $\Gamma(f)$ are typically singular with respect to the reference measure $\mu$, see \cite{BBST, Hino05}. In this case we use a different measure. A nonnegative Radon measure $m$ is called \emph{energy dominant} for $(\mathcal{E},\mathcal{F})$ if for any $f\in\mathcal{F}$ the energy measure $\Gamma(f)$ is absolutely continuous with respect to $m$. See \cite{Hino08, Hino10, Hino13}. 

Recall that a special standard core for the Dirichlet form $(\mathcal{E},\mathcal{F})$ is a subspace of $\mathcal{F}\cap C(X)$ that is a dense subalgebra of $C(X)$ and dense in $\mathcal{F}$, having the Markov property (for smooth contractions), and which for every open $U\subset X$ and compact $K\subset U$ contains a cut-off function
$0\leq \varphi \leq 1$, supported in $U$ and equal to one on $K$. See \cite[Section 1.1]{FOT94}. Our first basic assumption is as follows.

\begin{assumption}\label{A:Adense}
Assume that $\mathcal{A}\subset \mathcal{F}\cap C(X)$ is a special standard core for $(\mathcal{E},\mathcal{F})$ and $m$ is an energy dominant measure for $(\mathcal{E},\mathcal{F})$ such that 
\begin{equation}\label{E:boundeddensities}
\frac{d\Gamma(f,g)}{dm}\in L^\infty(X,m)\ \ \text{ for all $f,g\in \mathcal{A}$.    }
\end{equation}
\end{assumption}

As $m$ is fixed, we write $x\mapsto \Gamma(f,g)(x)$ for the densities $\frac{d\Gamma(f,g)}{dm}$.

\begin{remark}
Such $\mathcal{A}$ and $m$ always exist: 
By folklore arguments we can find an energy dominant measure $m$ and a countable family of functions $\left\lbrace \varphi_k\right\rbrace_{k=1}^\infty\subset \mathcal{F}\cap C(X)$ that is dense in $\mathcal{F}$, separates the points of $X$ and is such that  $\frac{d\Gamma(\varphi_k)}{dm}\in L^\infty(X,m)\ \text{ for all $k$}$, see for instance \cite[Lemma 2.1]{HKT}. Let $\mathcal{A}_0$ denote the algebra generated by $\left\lbrace \varphi_k\right\rbrace_{k=1}^\infty$ and the constants. By Stone-Weierstrass $\mathcal{A}_0$ is uniformly dense in $C(X)$. Now let $\mathcal{A}$ be the algebra of functions $f=F\circ g$, where $g\in\mathcal{A}$ and $F:\mathbb{R}\to\mathbb{R}$ is a Lipschitz function with  $F(0)=0$. By the arguments of \cite[Lemma 1.4.2 and Problem 1.4.1]{FOT94}
the algebra $\mathcal{A}$ provides a special standard core for $(\mathcal{E},\mathcal{F})$. Combining \cite[Corollary 7.1.2]{BH91} and \cite[Corollary 8.3]{H14a} we obtain (\ref{E:boundeddensities}).
\end{remark}

We consider the \emph{Hilbert space $(\mathcal{H},\left\langle\cdot,\cdot\right\rangle_{\mathcal{H}})$ of $L^2$- differential $1$-forms associated with $(\mathcal{E},\mathcal{F})$} and the corresponding first order derivation $\partial_{0}:\mathcal{A}\to \mathcal{H}$
as introduced in \cite{CS03} and studied in \cite{CS09, HKT, HRT, HT, IRT}. This is a generalized $L^2$-theory of $1$-forms. In particular, $\partial_0$ extends to a closed unbounded operator from $L^2(X,\mu)$ into $\mathcal{H}$. There are a measurable field $(\mathcal{H}_x)_{x\in X}$ of Hilbert spaces $(\mathcal{H}_x, \left\langle\cdot,\cdot\right\rangle_{\mathcal{H}_x})$ on $X$ and projections $\omega\mapsto \omega_x$ from $\mathcal{H}$ into $\mathcal{H}_x$ such that 
\begin{equation}\label{E:directint}
\left\langle \omega, \eta\right\rangle_{\mathcal{H}}=\int_X\left\langle \omega_x, \eta_x\right\rangle_{\mathcal{H}_x}\:m(dx), \ \ \omega,\eta\in\mathcal{H}.
\end{equation}
Up to isometry, $\mathcal{H}$ equals the direct integral $L_2(X,(\mathcal{H}_x)_{x\in X},m)$ of the Hilbert spaces $\mathcal{H}_x$, and we have
\begin{equation}\label{E:SPfiber}
\left\langle f_1\otimes g_1, f_2\otimes g_2\right\rangle_{\mathcal{H}_x}:=g_1(x)g_2(x)\Gamma(f_1, f_2)(x)
\end{equation}
for $m$-a.e. $x\in X$.
In particular, $\left\langle (\partial_0f)_x, (\partial_0 g)_x\right\rangle_{\mathcal{H}_x}=\Gamma(f,g)(x)$
for any $f,g\in\mathcal{A}$ and $m$-a.e. $x\in X$. The fibers $\mathcal{H}_x$ depend on the choice of $m$, but the space $\mathcal{H}$ does not. See for instance \cite[Section 2]{HRT}. The essential supremum $m$-$\esssup_{x\in X}\dim\mathcal{H}_x$ of the fiber dimensions $\dim \mathcal{H}_x$ is referred to as the \emph{martingale dimension} or \emph{index} of $(\mathcal{E},\mathcal{F})$. It is independent of the choice of $m$, \cite{Hino08, Hino10}.

\begin{examples}
In the Riemannian situation of Example \ref{Ex:rmf} we have  
$\mathcal{H}_x=T_x^\ast M$, and $\mathcal{H}$ is the space $L^2(M, T^\ast M, dvol)$ of $L^2$-differential $1$-forms on $M$. The operator $\partial_{0}$ coincides with the exterior derivation $d_{0}$ (in $L_2$-sense) taking functions into $1$-forms. If the manifold is $n$-dimensional, the martingale dimension of the Dirichlet form in \ref{Ex:rmf} is $n$. In particular, the $L^2$-space of $(n+1)$-forms is trivial, $L^2(M, \Lambda^{n+1}T^\ast M, dvol)=\left\lbrace 0\right\rbrace$.
\end{examples}

\begin{remark}
Hino showed in \cite[Theorem 4.4]{Hino08}, \cite[Theorem 4.5]{Hino10} and \cite[Theorem 4.10]{Hino13} that the martingale dimension for generic diffusions on p.c.f. self-similar fractals equals one. In \cite[Theorem 4.16]{Hino13} he showed that for generic diffusions on self-similar (generalized) Sierpinski carpets the martingale dimension is bounded from above by the spectral dimension.
\end{remark}

On a manifold the classical exterior derivation $d_1$ takes a $1$-forms $fd_0g$ into the $2$-forms $d_1(fd_0g)=d_0f\wedge d_0g$. In the present situation the space $\mathcal{F}\otimes \mathcal{A}$, spanned by elements of form $f\partial_0 g$ with $f\in\mathcal{A}$ and $g\in\mathcal{F}$, is a dense subspace of $\mathcal{H}$. Usual Hilbert space definitions lead to the \emph{space $L^2(X,(\hat{\Lambda}^2\mathcal{H}_x)_{x\in X}, m)$ of generalized $L^2$-differential forms of order $2$ associated with $(\mathcal{E},\mathcal{F})$ and $m$}, and we can introduce a generalization
\begin{equation}\label{E:welldef}
\partial_1: \mathcal{F}\otimes\mathcal{A}\to L^2(X,(\hat{\Lambda}^2\mathcal{H}_x)_{x\in X}, m)
\end{equation}
of the exterior derivation $d_1$ such that $\partial_1(f\partial_0 g)=\partial_0f\wedge \partial_0 g$. 

\begin{examples}\label{Ex:inidom}
In the Riemannian case $L^2(X,(\hat{\Lambda}^2\mathcal{H}_x)_{x\in X}, m)$ is $L^2(M, \Lambda^2 T^\ast M, dvol)$ and $\partial_1$ agrees with $d_1$.
\end{examples}

We make an additional assumption on the compact metric space $X$. Given a compact set $F\subset X$, we say that a function is locally constant quasi everywhere on an open cover of $F$ if it is constant quasi everywhere on the union of all sets in the cover. We write $\mathcal{S}^F$ for the space of all $f\in \mathcal{F}\cap L^\infty(X,m)$ for which there exists a finite open cover of $F$ such that the quasi-continuous version $\widetilde{f}$ of $f$ is locally constant quasi everywhere on this cover.

\begin{assumption}\label{A:carpetlike}
There is a topological base $\mathcal{O}$ for $X$, stable under taking finite unions, such that
\begin{enumerate}
\item[(i)] for any finite collection of pairwise disjoint base sets $O_1,...,O_M\in \mathcal{O}$ and any $g\in\mathcal{F}$ we can find a sequence of functions $(g_k)_k\subset \mathcal{S}^{\bigcup_{i=1}^M \partial O_i}$ such that $\lim_k\mathcal{E}(g_k-g)=0$,
\item[(ii)] For any $\delta>0$ there exist a finite collection of pairwise disjoint base sets $O_1,...,O_M \in\mathcal{O}$ such that $\diam (O_i)<\delta$, $i=1,...,M$ and $\bigcup_{i=1}^M \overline{O_i}=X$.
\end{enumerate}
\end{assumption}

\begin{theorem}\label{T:general}
Suppose that Assumptions \ref{A:Adense} and \ref{A:carpetlike} are satisfied. Then either the martingale dimension of $(\mathcal{E},\mathcal{F})$ is one or the derivation $\partial_{1}$ in (\ref{E:welldef}) is not closable. 
\end{theorem}

\begin{remark}\label{R:connect}\mbox{}
\begin{enumerate}
\item[(i)] In some sense the abstract conditions in Assumption \ref{A:carpetlike} require $X$ to be structured like a Sierpinski carpet. In particular, they are valid for the generalized Sierpinski carpets $S$ satisfying (\ref{E:condan}) in Section \ref{S:sc}: By enscribing a diagonal in each set $S_{n,k}$ we can triangulate $S$. Choosing $n$ larger and larger, we can approximate a given $C^1(\mathbb{R}^2)$-function on $S$ in energy by functions that are linear on these triangles. Given such a piecewise linear function $g$ we can then cover $S$ by $M=M(n')$ sets $\overline{O_i}=S_{n',k}$ and apply a slight modification of the proof of Theorem \ref{thm01} to see that if $n'$ is chosen large enough, we can find a function $\widetilde{g}\in\mathcal{S}^{\bigcup_{i=1}^M \partial O_i}$ that is arbitrarily close to $g$ in energy. 
\item[(ii)] For any compact topologically one-dimensional metric space $X$ that satisfies Assumption \ref{A:carpetlike} the locally harmonic $1$-forms are dense in the orthogonal complement $\mathcal{H}^1(X)$ of $Im\:\partial$ in $\mathcal{H}$, see \cite[Theorem 4.2]{HT}. According to (i) above, this result holds in particular for the generalized Sierpinski carpets $S$ satisfying (\ref{E:condan}). However, one would expect this result to be true only if there are no nontrivial $2$-forms, i.e. if $L^2(X,(\hat{\Lambda}^2\mathcal{H}_x)_{x\in X}, m)=\left\lbrace 0\right\rbrace$, which is not the case for these carpets. If $\partial_1$ were closable, this would produce a contradiction. Theorem \ref{T:general} excludes this possibility.
\end{enumerate}
\end{remark}

Before giving a proof of Theorem \ref{T:general} in Section \ref{S:proof} we provide details for the operator in (\ref{E:welldef}). 

\section{Algebraic definitions, energy norms and wedge products}\label{S:alg}

We consider some standard items of the theory of universal graded differential algebras, see \cite{Ei99, GVF}. By $\overline{\mathcal{A}}:=\mathcal{A}/\mathbb{R}$ we denote the algebra obtained from $\mathcal{A}$ by factoring out constants, and we will make use of a similar notation for other spaces. Notationally we do not distinguish between an element and its class. We consider the space $\overline{\mathcal{A}}\otimes\mathcal{A}$ spanned by tensors $f\otimes g$, by $(f\otimes g)(x,y)=f(x)g(y)$ they may be viewed as elements of $C(X\times X)$. In order to follow the notation of \cite{CS03, CS09, HT, IRT}, the notation used here deviates slightly from the one in \cite[Section 8.1]{GVF}, but structurally the definitions are the same. Right and left actions of $\mathcal{A}$ on $\overline{\mathcal{A}}\otimes\mathcal{A}$ can be defined as the linear extensions of 
\begin{equation}\label{E:actions}
(f\otimes g)h:=f\otimes (gh)\ \text{ and }\ h(f\otimes g):=(fh)\otimes g-h\otimes (fg).
\end{equation}
The definition
\begin{equation}\label{E:derivation}
\partial_0 f:=f\otimes \mathbf{1},\ \ f\in\mathcal{A},
\end{equation}
yields a linear operator $\partial_0:\mathcal{A}\to\overline{\mathcal{A}}\otimes\mathcal{A}$ that satisfies a \emph{product rule},
\begin{equation}\label{E:Leibniz}
\partial_0(fg)=f\partial_0 g +(\partial_0 f)g,\ \ f,g\in\mathcal{A}.
\end{equation}
The space $\overline{\mathcal{A}}\otimes\overline{\mathcal{A}}\otimes\mathcal{A}$ can be equipped with right and left actions of $\mathcal{A}$ by linearly extending
\begin{equation}\label{E:actions2r}
(f\otimes g\otimes h)u:=f\otimes g\otimes (hu)
\end{equation}
and
\begin{equation}\label{E:actions2l}
u(f\otimes g\otimes h):=u\otimes f\otimes (gh)-u\otimes (fg)\otimes h + (fu)\otimes g\otimes h.
\end{equation}
A product of two elements of $\overline{\mathcal{A}}\otimes\mathcal{A}$ is defined by
\begin{equation}\label{E:product}
(f\otimes g)(u\otimes v):=f\otimes (g(u\otimes v)).
\end{equation}
Similarly as in (\ref{E:derivation})
\begin{equation}\label{E:partial1}
\partial_1(f\otimes g):=f\otimes g\otimes \mathbf{1}
\end{equation}
defines a linear operator $\partial_1: \overline{\mathcal{A}}\otimes\mathcal{A}\to\overline{\mathcal{A}}\otimes
\overline{\mathcal{A}}\otimes\mathcal{A}$. We observe the identity $\partial_1\partial_0=0$.

\begin{remark}
One can also deduce a \emph{graded product rule} that yields
\begin{equation}\label{E:gradprod}
\partial_1(h\omega)=h\partial_1\omega-(\partial_0 h)\omega\ \ \text{ and}\ \ \partial_1(\omega h)=(\partial_1\omega)h+\omega\partial_0 h
\end{equation}
for tensors of form $\omega=f\otimes g$ and functions $h\in\mathcal{A}$. Formula (\ref{E:gradprod})  follows from  (\ref{E:actions}), (\ref{E:actions2r}), (\ref{E:actions2l}) and (\ref{E:product}).
\end{remark}

Suitable norms connect this algebraic point of view to the structure of $(\mathcal{E},\mathcal{F})$. On $\overline{\mathcal{A}}\otimes\mathcal{A}$ we consider the symmetric bilinear form given by the extension of
\begin{equation}\label{E:Hnorm}
\left\langle f_1\otimes g_1, f_2\otimes g_2\right\rangle_{\mathcal{H}}=\int_X g_1g_2\Gamma(f_1, f_2)\:dm.
\end{equation}
This form is nonnegative definite and induces a Hilbert seminorm $\left\|\cdot\right\|_{\mathcal{H}}$. Factoring $\overline{\mathcal{A}}\otimes\mathcal{A}$ by the kernel of this seminorm and completing yields a Hilbert space $(\mathcal{H},\left\|\cdot\right\|_{\mathcal{H}})$, referred to as the \emph{space of generalized $L^2$-differential $1$-forms} associated with $(\mathcal{E},\mathcal{F})$. Similarly as in \cite[Section 7]{HRT} one can verify the following statement.

\begin{lemma}\label{L:dense}
Under Assumption \ref{A:Adense} the space $\overline{\mathcal{C}}\otimes\mathcal{C}$ is a subspace of $\mathcal{H}$, and $\overline{\mathcal{A}}\otimes\mathcal{A}$ is dense in $\overline{\mathcal{C}}\otimes\mathcal{C}$ with respect to $\left\|\cdot\right\|_{\mathcal{H}}$.
\end{lemma}

Lemma \ref{L:dense} shows that indeed
$\mathcal{H}$ is the same Hilbert space of $1$-forms as defined by Cipriani and Sauvageot in \cite{CS03} and studied further by various authors, see for instance \cite{CGIS12, CS09, HKT, HRT, HT, IRT}. In particular, $\mathcal{H}$ does not depend on the choice of $m$.

By continuity the definitions (\ref{E:actions}) extend further to bounded linear actions of $\mathcal{A}$ on $\mathcal{H}$ with $\left\|\omega h\right\|_{\mathcal{H}}\leq \left\|h\right\|_{L^\infty(X,m)}\left\|\omega\right\|_{\mathcal{H}}$ and $\left\|h\omega\right\|_{\mathcal{H}}\leq \left\|h\right\|_{L^\infty(X,m)}\left\|\omega\right\|_{\mathcal{H}}$ 
for all $\omega\in\mathcal{H}$ and $h\in\mathcal{A}$. The strong locality of $(\mathcal{E},\mathcal{F})$ implies that the left and right actions of $\mathcal{A}$ on $\mathcal{H}$ agree.  For the derivation $\partial_0$ we have $\left\|\partial_0 f\right\|_{\mathcal{H}}^2=\mathcal{E}(f)$, $f\in\mathcal{A}$, and by the closedness of $(\mathcal{E},\mathcal{F})$ it extends to a closed unbounded linear operator $\partial_0: L^2(X,\mu)\to\mathcal{H}$ with domain $\mathcal{F}$. 

We endow the space $\overline{\mathcal{A}}\otimes\overline{\mathcal{A}}\otimes\mathcal{A}$ with the symmetric bilinear form defined by the bilinear extension of
\begin{equation}\label{E:H2norm}
\left\langle f_1\otimes g_1\otimes h_1, f_2\otimes g_2\otimes h_2\right\rangle_{\mathcal{H}^{(2)}}:=\int_Xh_1h_2\left(\Gamma(f_1,f_2)
\Gamma(g_1,g_2)-\Gamma(f_1,g_2)\Gamma(g_1,f_2)\right)\:dm.
\end{equation}
By the Cauchy-Schwarz inequality for $\Gamma$ this bilinear form is nonnegative definite and therefore produces a Hilbert seminorm $\left\|\cdot\right\|_{\mathcal{H}^{(2)}}$. Factoring out zero seminorm elements and completing we obtain a Hilbert space $(\mathcal{H}^{(2)},\left\|\cdot\right\|_{\mathcal{H}^{(2)}})$. Actions (\ref{E:actions2r}) and (\ref{E:actions2l}) extend to actions of $\mathcal{A}$ on $\mathcal{H}^{(2)}$ satisfying
$\left\|\omega h\right\|_{\mathcal{H}^{(2)}}\leq \left\|h\right\|_{L^\infty(X,m)}\left\|\omega\right\|_{\mathcal{H}^{(2)}}$ and $\left\|h\omega\right\|_{\mathcal{H}^{(2)}}\leq \left\|h\right\|_{L^\infty(X,m)}\left\|\omega\right\|_{\mathcal{H}^{(2)}}$
for any $\omega\in\mathcal{H}^{(2)}$ and $h\in\mathcal{A}$, and again the strong locality of $(\mathcal{E},\mathcal{F})$ can be used to verify that the right and the left action agree. In $\mathcal{H}^{(2)}$ the elements $h(\partial_1(g\partial_0 f))$ and $f\otimes g\otimes h$ agree for any  $f,g\in \overline{\mathcal{A}}$ and $h\in \mathcal{A}$ and their span is dense.

Let $(\mathcal{H}_x)_{x\in X}$ be the measurable field of Hilbert spaces $(\mathcal{H}_x, \left\langle\cdot,\cdot\right\rangle_{\mathcal{H}_x})$ as discussed in (\ref{E:directint}). We consider exterior products of the fibers $\mathcal{H}_x$, see for instance \cite[Section V.1]{Temam97}. For fixed $x\in X$ the tensor product $\omega_x^1\otimes \eta_x^1$ of two elements
$\omega_x^1$ and $\eta_x^1$ of $\mathcal{H}_x$ is defined as the bilinear form
$(\omega_x^1\otimes \eta_x^1)(\omega_x^2, \eta_x^2):=\left\langle \omega_x^1, \omega_x^2\right\rangle_{\mathcal{H}_x}\left\langle \eta_x^1, \eta_x^2\right\rangle_{\mathcal{H}_x}$, $\omega_x^2, \eta_x^2\in\mathcal{H}_x$.
As usual, we will denote the span of all such tensor products by $\bigotimes^2 \mathcal{H}_x$ and write $\Lambda^2 \mathcal{H}_x$ for the subspace of $\bigotimes^2\mathcal{H}_x$ spanned by the elements of type 
$\omega_x\wedge \eta_x:=\omega_x\otimes \eta_x-\eta_x\otimes \omega_x$. 
We endow $\Lambda^2 \mathcal{H}_x$ with its standard scalar product, defined as the bilinear extension of
\[\left\langle \omega_x^1\wedge\eta_x^1, \omega_x^2\wedge\eta_x^2\right\rangle_{\hat{\Lambda}^2\mathcal{H}_x}:=
\left\langle\omega_x^1, \omega_x^2\right\rangle_{\mathcal{H}_x}\left\langle \eta_x^1, \eta_x^2\right\rangle_{\mathcal{H}_x}-\left\langle \omega_x^1,\eta_x^2\right\rangle_{\mathcal{H}_x}\left\langle \eta_x^1,\omega_x^2\right\rangle_{\mathcal{H}_x}.\]
Let $\hat{\Lambda}^2\mathcal{H}_x$ denote the completion of $\Lambda^2 \mathcal{H}_x$ in this scalar product. It is not difficult to see that $(\hat{\Lambda}^2\mathcal{H}_x)_{x\in X}$ is again a measurable field of Hilbert spaces on $X$, and we consider its direct integral 
$L^2(X,(\hat{\Lambda}^2\mathcal{H}_x)_{x\in X}, m)$, equipped with the natural scalar product
\[(\xi,\zeta)\mapsto \int_X\left\langle \xi_x,\zeta_x\right\rangle_{\hat{\Lambda}^2\mathcal{H}_x} m(dx).\]
We introduce exterior products in the $L_2$-sense.

\begin{lemma}
Let $f_1, f_2\in\mathcal{F}$ be functions such that at least one of them is a member of
$\mathcal{A}$ and let $g_1,g_2\in L^\infty(X,m)$. Then 
\[(g_1\partial_0 f_1)\wedge (g_2\partial_0 f_2):=((g_1(x)(\partial_0 f_1)_x)\wedge (g_2(x)(\partial_0f_2)_x))_{x\in X}\]
is a member of $L^2(X,(\hat{\Lambda}^2\mathcal{H}_x)_{x\in X}, m)$.
\end{lemma}

\begin{proof} Suppose $f_1\in\mathcal{A}$. Formula (\ref{E:SPfiber}) and Cauchy-Schwarz applied to the fibers $\mathcal{H}_x$ yield 
\begin{align}
&
\left\|(g_1\partial_0 f_1)\wedge (g_2\partial_0 f_2)\right\|_{L^2(X,(\hat{\Lambda}^2\mathcal{H}_x)_{x\in X}, m)}^2
\\&
\leq 4\int_X\left\|(g_1(x)(\partial_0f_1)_x\right\|_{\mathcal{H}_x}^2\left\|g_2(x)(\partial_0 f_2)_x\right\|_{\mathcal{H}_x}^2 m(dx)\notag
\\&
=4\left\|g_1\right\|_{L^\infty(X,m)}^2\left\|\Gamma(f_1)\right\|_{L^\infty(X,m)}\left\|g_2\right\|_{L^\infty(X,m)}^2\mathcal{E}(f_2)\notag
<+\infty.\notag
\end{align}
\end{proof}

\begin{remark}
We observe that by construction $g_1\partial_0 f_1\wedge g_2\partial_0 f_2=g_1g_2\partial_0 f_1\wedge \partial_0 f_2$ and $g_1\partial_0 f_1\wedge g_2\partial_0 f_2=-g_2\partial_0 f_2\wedge g_1\partial_0 f_1$.
\end{remark}

Given $f\in\mathcal{F}$ and $ g\in\mathcal{A}$ we define
\begin{equation}\label{E:extendpartial}
\widetilde{\partial_1}(g\partial_0 f):=\partial_0g\wedge \partial_0f 
\end{equation}
and consider $\widetilde{\partial_1}$ as a densely defined unbounded operator from $\widetilde{\partial_1}:L^2(X,(\mathcal{H}_x)_{x\in X}, m)$ into $L^2(X,(\hat{\Lambda}^2\mathcal{H}_x)_{x\in X}, m)$
with initial domain $\mathcal{F}\otimes\mathcal{A}$. Note that we have $\overline{\mathcal{A}}\otimes\mathcal{A}\subset \mathcal{F}\otimes \mathcal{A}$.

This construction agrees with the preceding tensor product construction, and
in particular, $\widetilde{\partial_1}$ as in (\ref{E:extendpartial}) may be seen as an extension of $\partial_1$ as defined in (\ref{E:partial1}). To see this, define a linear map $\iota: \overline{\mathcal{A}}\otimes \overline{\mathcal{A}}\otimes\mathcal{A}\to L^2(X,(\hat{\Lambda}^2\mathcal{H}_x)_{x\in X}, m)$ by extending
\[\iota(h\partial_1(g\partial_0 f)):=h\partial_0 g\wedge \partial_0 f.\]
The following is a a consequence of (\ref{E:SPfiber}) and (\ref{E:H2norm}). 

\begin{corollary}
The map $\iota$ extends to an isometric isomorphism from the space $\mathcal{H}^{(2)}$ onto the space $L^2(X,(\hat{\Lambda}^2\mathcal{H}_x)_{x\in X}, m)$, and we have $\iota\circ \partial_1 = \widetilde{\partial_1}$ on $\overline{\mathcal{A}}\otimes\mathcal{A}$.
\end{corollary}

To $L^2(X,(\hat{\Lambda}^2\mathcal{H}_x)_{x\in X}, m)$ we refer as the \emph{space of generalized $L^2$-differential $2$-forms} associated with $(\mathcal{E},\mathcal{F})$ and $m$. We denote $\widetilde{\partial_1}$ again by $\partial_1$ and consider $\partial_1$ as an unbounded linear operator from $\mathcal{H}$ into $L^2(X,(\hat{\Lambda}^2\mathcal{H}_x)_{x\in X}, m)$
with dense initial domain $\mathcal{F}\otimes\mathcal{A}$.

\section{Proof of Theorem \ref{T:general}}\label{S:proof}

We proceed to a proof of Theorem \ref{T:general}. As no confusion can occur in this section we now write $\partial$ for both $\partial_0$ and $\partial_1$. Assume there is a non-zero element $\partial f\wedge \partial g$ with certain $f,g\in \mathcal{A}$. We construct a sequence of $1$-forms $(\omega_n)_n\subset\mathcal{H}$ that converges to zero in $\mathcal{H}$ but is such that the sequence $(\partial \omega_n)_n$ of its images under $\partial$ approximates $\partial f\wedge \partial g$ in $L^2(X,(\hat{\Lambda}^2\mathcal{H}_x)_{x\in X},m)$. This implies Theorem \ref{T:general}.

By compactness for any $n$ there exists some $\delta_n>0$ such that $|f(x)-f(y)|<\frac1n$ for all $x,y\in X$ with $\varrho(x,y)<\delta_n$. Let $O_1^{(n)}, ..., O_{M_n}^{(n)}$ be a finite collection of pairwise disjoint open sets as in Assumption \ref{A:carpetlike} such that $\diam O_i^{(n)}<\delta_n$, $\bigcup_{i=1}^{M_n} \overline{O_i^{(n)}}=X$ and let $g_n\in\mathcal{S}^{\bigcup_{i=1}^{M_n} \partial O_i^{(n)}}$ be functions 
satisfying
\begin{equation}\label{E:energyboundgn}
\mathcal{E}(g-g_n)<2^{-n}.
\end{equation} 
By $W^{(n)}$ we denote the union of a sufficiently small finite open cover of $\bigcup_{i=1}^{M_n} \partial O_i^{(n)}$ on which $g_n$ is locally constant quasi everywhere. Let $\varphi_n\in \mathcal{A}$ be a function  supported in $W^{(n)}$ such that $0\leq \varphi_n\leq 1$ and $\varphi_n\equiv 1$ on $\bigcup_{i=1}^{M_n} \partial O_i^{(n)}$. For any $i=1,..., M_n$ fix a point $x_i^{(n)}\in O_i^{(n)}$. We define functions $f_n\in\mathcal{A}$ by
\[f_n(x):=\sum_{i=1}^{M_n} (1-\varphi_n(x))(f(x)-f(x_i)), \ \ x\in X.\]
They satisfy 
\begin{equation}\label{E:supboundfn}
\left\|f_n\right\|_{\sup}<\frac1n
\end{equation}
and $f(x)-f_n(x)=f(x_i)$ for all $x\in O_i^{(n)}\setminus W^{(n)}$. We define the $1$-forms
$\omega_n:=f_n\partial g_n$.

\begin{lemma} The sequence $(\omega_n)_n$ converges to zero in $\mathcal{H}$, and
the sequence $(\partial \omega_n)_n$ converges to $\partial f\wedge \partial g$ in $L^2(X,(\hat{\Lambda}^2\mathcal{H}_x)_{x\in X}, m)$.
\end{lemma}

\begin{proof}
The first statement follows from (\ref{E:energyboundgn}) and (\ref{E:supboundfn}). For the second, note that
\begin{align}
\left\|\partial f\wedge \partial g\right. &\left. -\partial f\wedge \partial g_n\right\|^2_{L^2(X,(\hat{\Lambda}^2\mathcal{H}_x)_{x\in X}, m)}\notag\\
&=\int_X\left(\left\|(\partial f)_x\right\|_{\mathcal{H}_x}^2\left\|(\partial(g-g_n))_x\right\|_{\mathcal{H}_x}^2-\left\langle (\partial f)_x,(\partial(g-g_n))_x\right\rangle_{\mathcal{H}_x}^2\right)m(dx)
\notag\\
&\leq 2\left\|\Gamma(f)\right\|_{L^\infty(X,m)}^2\: \mathcal{E}(g-g_n),\notag
\end{align}
what converges to zero by (\ref{E:energyboundgn}). On the other hand we observe
\begin{align}
\left\|\partial f\wedge\partial g_n\right. &\left. -\omega_n\right\|_{L^2(X,(\hat{\Lambda}^2\mathcal{H}_x)_{x\in X}, m)}^2\notag\\
&=\int_X\left(\left\|(\partial (f-f_n))_x\right\|_{\mathcal{H}_x}^2\left\|(\partial g_n)_x\right\|_{\mathcal{H}_x}^2-\left\langle (\partial(f-f_n))_x,(\partial g_n)_x\right\rangle_{\mathcal{H}_x}^2\right)m(dx)\notag\\
&\leq 2\int_{W^{(n)}}\Gamma(f-f_n)\Gamma(g_n)\:dm+2\int_{X\setminus  W^{(n)}}\Gamma(f-f_n)\Gamma(g_n)\:dm\notag\\
&=0,\notag
\end{align}
because $g_n$ is locally constant q.e. on $W^{(n)}$, what implies $\Gamma(g_n)(x)=0$ for $m$-a.e. $x\in W^{(n)}$, and $f-f_n$ is locally constant on $X\setminus W^{(n)}$, what implies $\Gamma(f-f_n)(x)=0$ for $m$-a.e. $x\in X\setminus W^{(n)}$. 
\end{proof}

\end{document}